\documentclass{amsart}
\usepackage{amssymb}
\usepackage{amsmath}
\usepackage{amsfonts}

\newtheorem{theorem}{Theorem}
\theoremstyle{plain}

\newtheorem{corollary}{Corollary}

\newtheorem{lemma}{Lemma}

\newtheorem{proposition}{Proposition}
\newtheorem{remark}{Remark}

\numberwithin{equation}{section}
\input{tcilatex}

\begin{document}
\title{CLASSIFICATION OF $3$-DIMENSIONAL CONFORMALLY FLAT
QUASI-PARA-SASAKIAN MANIFOLDS}
\author{I. K\"{u}peli Erken}
\address{Faculty of Engineering and Natural Sciences, Department of
Mathematics, Bursa Technical University, Bursa, TURKEY}
\email{irem.erken@btu.edu.tr}
\date{05.03.2018}

\begin{abstract}
The object of the present paper is to study $3$-dimensional conformally flat
quasi-Para-Sasakian manifolds. First, the necessary and sufficient
conditions are provided for $3$-dimensional quasi-Para-Sasakian manifolds to
be conformally flat. Next, a characterization of $3$-dimensional conformally
flat quasi-Para-Sasakian manifold with $\beta =$const. is given.
\end{abstract}

\maketitle

% no mail address for me

\section{I\textbf{ntroduction}}

Almost paracontact metric structures are the natural odd-dimensional
analogue to almost paraHermitian structures, just like almost contact metric
structures correspond to the almost Hermitian ones. The study of almost
paracontact geometry was introduced by Kaneyuki and Williams in \cite%
{kaneyuki1} and then it was continued by many other authors. A systematic
study of almost paracontact metric manifolds was carried out in paper of
Zamkovoy, \cite{Za}. Comparing with the huge literature in almost contact
geometry, it seems that there are necessary new studies in almost
paracontact geometry. Therefore, paracontact metric manifolds have been
studied in recent years by many authors, emphasizing similarities and
differences with respect to the most well known contact case. Interesting
papers connecting these fields are (see, for example, \cite{DACKO}, \cite%
{Biz}, \cite{Welyczko}, \cite{Za}, and references therein).

Z. Olszak studied normal almost contact metric manifolds of dimension $3$ 
\cite{olszak}. He derive certain necessary and sufficient conditions for an
almost contact metric structure on manifold to be normal and curvature
properties of such structures and normal almost contact metric structures on
a manifold of constant curvature are studied. Recently, J. We\l yczko
studied curvature and torsion of Frenet-Legendre curves in $3$-dimensional
normal almost paracontact metric manifolds \cite{Welyczko1} and then normal
almost paracontact metric manifolds are studied by \cite{bejan}, \cite{irem1}%
, \cite{irem2}.

The notion of quasi-Sasakian manifolds, introduced by D. E. Blair in \cite%
{BL1}, unifies Sasakian and cosymplectic manifolds. By definition, a
quasi-Sasakian manifold is a normal almost contact metric manifold whose
fundamental $2$-form $\Phi :=g(\cdot ,\phi \cdot )$ is closed.
Quasi-Sasakian manifolds can be viewed as an odd-dimensional counterpart of
Kaehler structures. These manifolds studied by several authors(e.g. \cite%
{KANEMAKI}, \cite{OL1}, \cite{ols2}, \cite{TAN1}).

Although quasi-Sasakian manifolds were studied by several different authors
and are considered a well-established topic in contact Riemannian geometry,
To the authors knowledge, there do not exist any study about
quasi-Para-Sasakian manifolds.

Motivated by these considerations, in \cite{irem3}, the author make the
first contribution to investigate basic properties and general curvature
identities of quasi-Para-Sasakian manifolds.

In this paper, we study $3$-dimensional conformally flat quasi-Para-Sasakian
manifolds.

The paper is organized as follows:

Section $2$ is preliminary section, where we recall the definition of almost
paracontact metric manifold and quasi-Para-Sasakian manifolds.

In Section $3$, we mainly proved that for a $3$-dimensional
quasi-Para-Sasakian manifold $M$ with $\beta =$const., the followings are
equivalent.

$i)$ $M$\ is locally symmetric.

$ii)$ $M$\ is conformally flat and its scalar curvature $\tau $\ is const.,

$iii)$ $M$\ is conformally flat and $\beta $\ is const.,

$iv)\bullet $ If $\beta =0$, then $M$ is a paracosymplectic manifold which
is locally a product of the real line $R$ and a $2$-dimensional
para-Kaehlerian manifold or

$\bullet $If $\beta \neq 0$, then $M$ is of constant negative curvature and
the quasi-para-Sasakian structure can be obtained by a homothetic
deformation of a para-Sasakian structure.

\section{Preliminaries}

Let $M$ be a $(2n+1)$-dimensional differentiable manifold and $\phi $ is a $%
(1,1)$ tensor field, $\xi $ is a vector field and $\eta $ is a one-form on $%
M.$ Then $(\phi ,\xi ,\eta )$ is called an \textit{almost paracontact
structure} on $M$ if

\begin{itemize}
\item[(i)] $\phi^2 = Id-\eta\otimes\xi, \quad \eta(\xi)=1$,

\item[(ii)] the tensor field $\phi $ induces an almost paracomplex structure
on the distribution $D=$ ker $\eta ,$ that is the eigendistributions $D^{\pm
},$ corresponding to the eigenvalues $\pm 1$, have equal dimensions, $\text{%
dim}\,D^{+}=\text{dim}\,D^{-}=n$.
\end{itemize}

The manifold $M$ is said to be an \textit{almost paracontact manifold} if it
is endowed with an almost paracontact structure \cite{Za}.

Let $M$ be an almost paracontact manifold. $M$ will be called an \textit{%
almost} \textit{paracontact metric manifold} if it is additionally endowed
with a pseudo-Riemannian metric $g$ of a signature $(n+1,n)$, i.e.%
\begin{equation}
g(\phi X,\phi Y)=-g(X,Y)+\eta (X)\eta (Y).  \label{1}
\end{equation}

For such manifold, we have 
\begin{equation}
\eta (X)=g(X,\xi ),\text{ }\phi (\xi )=0,\text{ }\eta \circ \phi =0.
\label{2}
\end{equation}

Moreover, we can define a skew-symmetric tensor field (a $2$-form) $\Phi $ by%
\begin{equation}
\Phi (X,Y)=g(X,\phi Y),  \label{3}
\end{equation}%
usually called \textit{fundamental form}.

For an almost paracontact manifold, there exists an orthogonal basis $%
\{X_{1},\ldots ,X_{n},Y_{1},\ldots ,$ $Y_{n},\xi \}$ such that $%
g(X_{i},X_{j})=\delta _{ij}$, $g(Y_{i},Y_{j})=-\delta _{ij}$ and $Y_{i}=\phi
X_{i}$, for any $i,j\in \left\{ 1,\ldots ,n\right\} $. Such basis is called
a $\phi $\textit{-basis}.

On an almost paracontact manifold, one defines the $(1,2)$-tensor field $%
N^{(1)}$ by%
\begin{equation}
N^{(1)}(X,Y)=\left[ \phi ,\phi \right] (X,Y)-2d\eta (X,Y)\xi ,  \label{nijen}
\end{equation}%
where $\left[ \phi ,\phi \right] $ is the \textit{Nijenhuis torsion} of $%
\phi $%
\begin{equation*}
\left[ \phi ,\phi \right] (X,Y)=\phi ^{2}\left[ X,Y\right] +\left[ \phi
X,\phi Y\right] -\phi \left[ \phi X,Y\right] -\phi \left[ X,\phi Y\right] .
\end{equation*}

If $N^{(1)}$ vanishes identically, then the almost paracontact manifold
(structure) is said to be \textit{normal} \cite{Za}. The normality condition
says that the almost paracomplex structure $J$ defined on $M\times 
%TCIMACRO{\U{211d} }%
%BeginExpansion
\mathbb{R}
%EndExpansion
$%
\begin{equation*}
J(X,\lambda \frac{d}{dt})=(\phi X+\lambda \xi ,\eta (X)\frac{d}{dt}),
\end{equation*}%
is integrable.

If $d\eta (X,Y)=$ $g(X,\phi Y)$, then $(M,\phi ,\xi ,\eta ,g)$ is said to be 
\emph{paracontact metric manifold.} In a paracontact metric manifold one
defines a symmetric, trace-free operator $h=\frac{1}{2}{\mathcal{L}}_{\xi
}\phi $, where $\mathcal{L}_{\xi }$, denotes the Lie derivative. It is known 
\cite{Za} that $h$ anti-commutes with $\phi $ and satisfies $h\xi =0,$ tr$h=$%
tr$h\phi =0$ and $\nabla \xi =-\phi +\phi h,$ where $\nabla $ is the
Levi-Civita connection of the pseudo-Riemannian manifold $(M,g)$.

Moreover $h=0$ if and only if $\xi $ is Killing vector field. In this case $%
(M,\phi ,\xi ,\eta ,g)$ is said to be a \emph{K-paracontact manifold}.
Similarly as in the class of almost contact metric manifolds \cite{blair2}$,$
a normal almost paracontact metric manifold will be called \textit{%
para-Sasakian }if $\Phi =d\eta $ \cite{erdem}.

Now, we will give some results about $3$-dimensional quasi-para-Sasakian
manifolds that we will use next section.

\begin{proposition}
\cite{Welyczko1}\label{J3} For a $3$\textit{-dimensional almost paracontact
metric manifold }$M$\textit{\ the following three conditions are mutually
equivalent}
\end{proposition}

$(a)$ $M$ \textit{is normal,}

$(b)$\textit{\ there exist functions }$\alpha ,$\textit{\ }$\beta $\textit{\
on }$M$\textit{\ such that}

\begin{equation}
(\nabla _{X}\phi )Y=\beta (g(X,Y)\xi -\eta (Y)X)+\alpha (g(\phi X,Y)\xi
-\eta (Y)\phi X),  \label{N3}
\end{equation}

$(c)$\textit{\ there exist functions }$\alpha ,$\textit{\ }$\beta $\textit{\
on }$M$\textit{\ such that}%
\begin{equation}
\nabla _{X}\xi =\alpha (X-\eta (X)\xi )+\beta \phi X.  \label{N4}
\end{equation}

\begin{corollary}
\cite{irem1}\label{irem} For a normal almost paracontact metric structure $%
(\phi ,\xi ,\eta ,g)$ on $M$, we have $\nabla _{\xi }\xi =0$ and $d\eta
=-\beta \Phi $. \textit{The functions }$\alpha ,\beta $\textit{\ realizing (%
\ref{N3}) as well as (\ref{N4}) are given by }\cite{Welyczko1}%
\begin{equation}
2\alpha =\text{ Trace }\left\{ X\longrightarrow \nabla _{X}\xi \right\} ,%
\text{ \ }2\beta =\text{Trace }\left\{ X\longrightarrow \phi \nabla _{X}\xi
\right\} .  \label{N5}
\end{equation}
\end{corollary}

\begin{proposition}
\cite{Welyczko1}\label{J5} For a $3$-\textit{dimensional almost paracontact
metric\ manifold }$M$\textit{, the following three conditions are mutually
equivalent}
\end{proposition}

$(a)$ $M$ \textit{is quasi-para-Sasakian,}

$(b)$\textit{there exists a function }$\beta $\textit{\ on }$M$\textit{\
such that}%
\begin{equation}
(\nabla _{X}\phi )Y=\beta (g(X,Y)\xi -\eta (Y)X),  \label{N6}
\end{equation}

$(c)$\textit{there exists a function }$\beta $\textit{\ on }$M$\textit{\
such that}%
\begin{equation}
\nabla _{X}\xi =\beta \phi X.  \label{N7}
\end{equation}%
A $3$-dimensional normal almost paracontact metric manifold is

$\bullet $ paracosymplectic if $\alpha =\beta =0$ \cite{DACKO},

$\bullet $ quasi-para-Sasakian if and only if $\alpha =0$ and $\beta \neq 0$ 
\cite{erdem}, \cite{Welyczko1},

$\bullet $ $\beta $-para-Sasakian if and only if $\alpha =0$ and $\beta \neq
0$ and $\beta $ is constant, in particular, para-Sasakian if $\beta =-1$ 
\cite{Welyczko1}, \cite{Za},

$\bullet $ $\alpha $-para-Kenmotsu if $\alpha \neq 0$ and $\alpha $ is
constant and $\beta =0.$

Obviously, the class of para-Sasakian manifolds is contained in the class of
quasi-para-Sasakian manifolds. The converse does not hold in general. Also
in this context the para-Sasakian condition implies the $K$-paracontact
condition and the converse holds only in dimension $3$. A paracontact metric
manifold will be called\textit{\ paracosymplectic} if $d\Phi =0,$ $d\eta =0$ 
\cite{DACKO}, obviously, the class of paracosymplectic manifolds is
contained in the class of quasi-para-Sasakian manifolds.

\begin{theorem}
\cite{irem1}\label{pi}Let $(M,\phi ,\xi ,\eta ,g)$ be a $3$-dimensional
normal almost paracontact metric manifold. Then the following curvature
identities hold%
\begin{eqnarray}
&&R(X,Y)Z  \notag \\
&=&(2(\xi (\alpha )+\alpha ^{2}+\beta ^{2})+\frac{1}{2}\tau
)(g(Y,Z)X-g(X,Z)Y)  \notag \\
&&-(\xi (\alpha )+3(\alpha ^{2}+\beta ^{2})+\frac{1}{2}\tau )((g(Y,Z)\eta
(X)\xi -g(X,Z)\eta (Y)\xi  \notag \\
&&+\eta (Y)\eta (Z)X-\eta (X)\eta (Z)Y)+(\phi Z(\beta )-Z(\alpha ))(\eta
(Y)X-\eta (X)Y)  \notag \\
&&+(\phi Y(\beta )-Y(\alpha ))(\eta (Z)X-g(X,Z)\xi )  \label{R5} \\
&&-\left( \phi X(\beta )-X(\alpha )\right) (\eta (Z)Y-g(Y,Z)\xi )  \notag \\
&&+(\phi \text{grad}\beta +\text{grad}\alpha )(\eta (Y)g(X,Z)-\eta
(X)g(Y,Z)).  \notag
\end{eqnarray}%
\begin{eqnarray}
S(Y,Z) &=&-(\xi (\alpha )+\alpha ^{2}+\beta ^{2}+\frac{1}{2}\tau )g(\phi
Y,\phi Z)  \label{S2} \\
&&+\eta (Z)\left( \phi Y(\beta )-Y(\alpha )\right)  \notag \\
&&+\eta (Y)\left( \phi Z(\beta )-Z(\alpha )\right) -2(\alpha ^{2}+\beta
^{2})\eta (Y)\eta (Z),  \notag
\end{eqnarray}%
where $R$, $S$ and $\tau $ are resp. Riemannian curvature, Ricci tensor and
scalar curvature of $M$.
\end{theorem}

If we take $\alpha =0$ in Theorem \textit{\ref{pi}, }we get following

\begin{theorem}
\label{new}Let $(M,\phi ,\xi ,\eta ,g)$ be a $3$-dimensional
quasi-Para-Sasakian manifold. Then the following curvature identities hold%
\begin{eqnarray}
&&R(X,Y)Z  \notag \\
&=&(2\beta ^{2}+\frac{1}{2}\tau )(g(Y,Z)X-g(X,Z)Y)  \notag \\
&&-(3\beta ^{2}+\frac{1}{2}\tau )((g(Y,Z)\eta (X)\xi -g(X,Z)\eta (Y)\xi 
\notag \\
&&+\eta (Y)\eta (Z)X-\eta (X)\eta (Z)Y)+\phi Z(\beta )(\eta (Y)X-\eta (X)Y) 
\notag \\
&&+\phi Y(\beta )(\eta (Z)X-g(X,Z)\xi )  \notag \\
&&-\phi X(\beta )(\eta (Z)Y-g(Y,Z)\xi )  \notag \\
&&+(\phi \text{grad}\beta )(\eta (Y)g(X,Z)-\eta (X)g(Y,Z)).  \label{m1}
\end{eqnarray}%
\begin{eqnarray}
S(Y,Z) &=&(\beta ^{2}+\frac{1}{2}\tau )g(Y,Z)-(3\beta ^{2}+\frac{1}{2}\tau
)\eta (Y)\eta (Z)  \notag \\
&&+\eta (Y)\phi Z(\beta )+\eta (Z)\phi Y(\beta ).  \label{m2}
\end{eqnarray}%
where $R$, $S$ and $\tau $ are resp. Riemannian curvature, Ricci tensor and
scalar curvature of $M$.
\end{theorem}

\begin{remark}
In the proof of Theorem \textit{\ref{pi}, the author showed that }$\xi
(\beta )+2\alpha \beta =0$. Namely, for $3$-dimensional quasi-Para-Sasakian
manifolds,%
\begin{equation}
\xi (\beta )=0.  \label{ZETA}
\end{equation}
\end{remark}

\begin{theorem}
\cite{irem3}\label{mert}Let $(M^{2n+1},\phi ,\xi ,\eta ,g)$ be a \textit{%
quasi-para-Sasakian manifold} of constant curvature $K$. Then $K\leq 0.$
Furthermore,
\end{theorem}

\textit{\ \ \ \ \ \ \ }$\bullet $\textit{If }$K=0,$ \textit{the manifold is
paracosymplectic,}

\textit{\ \ \ \ \ \ \ }$\bullet $\textit{If }$K<0,$\textit{\ the structure }$%
(\phi ,\xi ,\eta ,g)$\textit{\ is obtained by a homothetic deformation of a
para-Sasakian structure on }$M^{2n+1}.$

\section{$3$-dimensional conformally flat quasi-Para-Sasakian manifolds}

For the conformal flatness, we will use linear $(1,1)$-tensor field $L$
defined by%
\begin{equation}
L=Q-(\frac{r}{4}),  \label{C0}
\end{equation}
where $S(X,Y)=g(QX,Y).$

From now on, we will use the notion $df(X)$ instead of $g(\func{grad}f,X)$.

\begin{lemma}
\label{L1}The linear operator $L$ of a $3$-dimensional quasi-Para-Sasakian
manifold is given by%
\begin{equation}
LY=\left( \frac{\tau }{4}+\beta ^{2}\right) Y-\left( 3\beta ^{2}+\frac{\tau 
}{2}\right) \eta (Y)\xi -\eta (Y)\phi \func{grad}\beta +d\beta (\phi Y)\xi .
\label{C1}
\end{equation}
\end{lemma}

\begin{proof}
By (\ref{m2})\textit{,} we obtain%
\begin{equation}
QY=\left( \frac{\tau }{2}+\beta ^{2}\right) Y-\left( 3\beta ^{2}+\frac{\tau 
}{2}\right) \eta (Y)\xi -\eta (Y)\phi \func{grad}\beta +d\beta (\phi Y)\xi .
\label{c1.1}
\end{equation}%
The requested equation comes from combining (\ref{C0}) and the above last
equation.
\end{proof}

From (\ref{C1})\textit{, }we have%
\begin{eqnarray*}
(\nabla _{X}L)Y &=&\nabla _{X}LY-L\nabla _{X}Y \\
&=&\left( \frac{d\tau (X)}{4}+2\beta d\beta (X)\right) Y-\left( 6\beta
d\beta (X)+\frac{d\tau (X)}{2}\right) \eta (Y)\xi \\
&&-\left( 3\beta ^{2}+\frac{\tau }{2}\right) \left( (\nabla _{X}\eta )(Y)\xi
+\eta (Y)\nabla _{X}\xi \right) \\
&&-(\nabla _{X}\eta )(Y)\phi \func{grad}\beta -\eta (Y)(\nabla _{X}\phi )%
\func{grad}\beta -\eta (Y)\phi \nabla _{X}\func{grad}\beta \\
&&+(\nabla _{X}d\beta )(\phi Y)\xi +d\beta ((\nabla _{X}\phi )Y)\xi +d\beta
(\phi Y)\nabla _{X}\xi .
\end{eqnarray*}%
If we use (\ref{N6}), (\ref{N7}) and\textit{\ }(\ref{ZETA}) in the last
equation, we can state following:

\begin{lemma}
\label{L2}The covariant derivative of the linear operator $L$ of a $3$%
-dimensional quasi-Para-Sasakian manifold is given by%
\begin{eqnarray}
(\nabla _{X}L)Y &=&\left( \frac{d\tau (X)}{4}+2\beta d\beta (X)\right)
Y-\left( 6\beta d\beta (X)+\frac{d\tau (X)}{2}\right) \eta (Y)\xi  \notag \\
&&-\beta \left( 3\beta ^{2}+\frac{\tau }{2}\right) \left( g(\phi X,Y)\xi
+\eta (Y)\phi X\right) -\beta g(\phi X,Y)\phi \func{grad}\beta  \notag \\
&&-\beta d\beta (X)\eta (Y)\xi -\eta (Y)\phi \nabla _{X}\func{grad}\beta
+(\nabla _{X}d\beta )(\phi Y)\xi  \notag \\
&&-\beta \eta (Y)d\beta (X)\xi +\beta d\beta (\phi Y)\phi X.  \label{C2}
\end{eqnarray}
\end{lemma}

\begin{lemma}
\label{L3} For the the function $\beta $ of $3$-dimensional
quasi-Para-Sasakian manifold, the following equality holds%
\begin{equation}
\nabla _{\xi }\func{grad}\beta =\beta \phi \func{grad}\beta .  \label{C3}
\end{equation}
\end{lemma}

\begin{proof}
By virtue of (\ref{ZETA})$,$ we have%
\begin{equation}
\lbrack \xi ,X](\beta )=\xi (X(\beta ))-X(\xi (\beta ))=g(\nabla _{\xi }%
\func{grad}\beta ,X)+g(\func{grad}\beta ,\nabla _{\xi }X).  \label{C4}
\end{equation}%
By (\ref{N7}), we get 
\begin{equation}
\lbrack \xi ,X](\beta )=g(\func{grad}\beta ,[\xi ,X])=g(\func{grad}\beta
,\nabla _{\xi }X)+\beta g(\phi \func{grad}\beta ,X).  \label{C5}
\end{equation}%
The proof comes from (\ref{C4}) and (\ref{C5}).
\end{proof}

For any point $p$ $\in U\subset M$, there exists a local orthonormal $\phi $%
-basis $\left\{ e_{1}=\phi e_{2},\text{ }e_{2}=\phi e_{1},\text{ }e_{3}=\xi
\right\} $, where $g(e_{1},e_{1})$ $=$ $-g(e_{2},e_{2})=g(e_{3},e_{3})=1$.

For the sake of shortness, we will give followings%
\begin{eqnarray*}
\tau _{i} &=&d\tau (e_{i}), \\
\beta _{i} &=&d\beta (e_{i}) \\
\beta _{ij} &=&(\nabla _{e_{i}}d\beta )(e_{j}), \\
L_{ij} &=&(\nabla _{e_{i}}L)e_{j} \\
\func{grad}\beta &=&\beta _{1}e_{1}+\beta _{2}e_{2}+\beta _{3}e_{3}, \\
\nabla _{e_{i}}\func{grad}\beta &=&\beta _{i1}e_{1}+\beta _{i2}e_{2}+\beta
_{i3}e_{3}
\end{eqnarray*}%
for $1\leq i,j\leq 3,$ where $\tau _{i},$ $\beta _{i},~\beta _{ij}$ are the
functions and $L_{ij}$ are the vector fields on $U$. Also, one can easily
see that $\beta _{ij}=\beta _{ji}$. If we use (\ref{c1.1}), (\ref{N7}) and (%
\ref{m2}) in the following well known formula for semi-Riemannian manifolds 
\begin{equation*}
trace\left\{ Y\rightarrow (\nabla _{Y}Q)X\right\} =\frac{1}{2}\nabla _{X}\tau
\end{equation*}%
we obtain%
\begin{equation}
\xi (\tau )=0.  \label{c10}
\end{equation}%
From (\ref{ZETA}) and (\ref{c10}), we obtain%
\begin{equation}
\beta _{3}=0,~~\tau _{3}=0.  \label{c11}
\end{equation}%
(\ref{C3}) implies that%
\begin{equation}
\beta _{13}=\beta _{31}=-\beta \beta _{2},~~~\beta _{23}=\beta _{32}=-\beta
\beta _{1},~~~\beta _{33}=0.  \label{c12}
\end{equation}

\begin{lemma}
\label{L4}For a $3$-dimensional quasi-Para-Sasakian manifold, we have%
\begin{eqnarray}
L_{ij}-L_{ji} &=&0~\text{\ for }1\leq i,j\leq 3\Leftrightarrow  \notag \\
\tau _{1} &=&-20\beta \beta _{1},~~~\tau _{2}=-20\beta \beta _{2},~~~\beta
_{12}=\beta _{21}=0,~~~\beta _{22}=-\beta _{11}=\beta \left( 3\beta ^{2}+%
\frac{\tau }{2}\right) .  \label{c13}
\end{eqnarray}
\end{lemma}

\begin{proof}
By direct computations, using (\ref{C2}), (\ref{C3}), (\ref{c11}) and (\ref%
{c12}), we derive%
\begin{eqnarray}
L_{12}-L_{21} &=&-\left( \frac{\tau _{2}}{4}+5\beta \beta _{2}\right)
e_{1}+\left( \frac{\tau _{1}}{4}+5\beta \beta _{1}\right) e_{2}  \notag \\
&&+(\beta _{11}-\beta _{22}+\beta (\tau +6\beta ^{2}))\xi .  \notag \\
L_{13}-L_{31} &=&\beta _{12}e_{1}+\left( -\beta _{11}-\beta \left( 3\beta
^{2}+\frac{\tau }{2}\right) \right) e_{2}  \notag \\
&&+\left( -\frac{\tau _{1}}{4}-5\beta \beta _{1}\right) \xi .  \notag \\
L_{23}-L_{32} &=&\left( \beta _{22}-\beta \left( 3\beta ^{2}+\frac{\tau }{2}%
\right) \right) e_{1}-\beta _{12}e_{2}  \notag \\
&&+\left( -\frac{\tau _{2}}{4}-5\beta \beta _{2}\right) \xi .  \label{c14}
\end{eqnarray}%
The proof follows from (\ref{c14}).
\end{proof}

We know that a semi-Riemannian manifold is conformally flat$\Leftrightarrow
(\nabla _{X}L)Y-(\nabla _{Y}L)X=0$, for any vector fields $X$ and $Y.$
Hence, we can say that a $3$-dimensional quasi-Para-Sasakian manifold is
conformally flat if and only if (\ref{c13}) holds. By (\ref{c13}), we can
give following result.

\begin{theorem}
\label{L5}A $3$-dimensional quasi-Para-Sasakian manifold is conformally flat
if and only if the function $\beta $ satisfies the followings%
\begin{eqnarray}
\tau +10\beta ^{2} &=&\text{const.,}  \notag \\
(\nabla _{X}d\beta )(Y) &=&-\beta \left( 3\beta ^{2}+\frac{\tau }{2}\right)
(g(X,Y)-\eta (X)\eta (Y))  \notag \\
&&-\beta \eta (X)d\beta (\phi Y)-\beta \eta (Y)d\beta (\phi X).  \label{C15}
\end{eqnarray}
\end{theorem}

\begin{theorem}
\label{L6}For a $3$-dimensional quasi-Para-Sasakian manifold $M$ with $\beta
=$const., the following assertions are equivalent to each other:
\end{theorem}

$i)$\textit{\ }$M$\textit{\ is locally symmetric.}

$ii)$\textit{\ }$M$\textit{\ is conformally flat and its scalar curvature }$%
\tau $\textit{\ is const.,}

$iii)$\textit{\ }$M$\textit{\ is conformally flat and }$\beta $\textit{\ is
const.,}

$iv)\bullet $\textit{\ If }$\beta =0$\textit{, then }$M$\textit{\ is a
paracosymplectic manifold which is locally a product of the real line }$R$%
\textit{\ and a }$2$\textit{-dimensional para-Kaehlerian manifold or}

$\bullet $\textit{\ If }$\beta \neq 0$\textit{, then }$M$\textit{\ is of
constant negative curvature and the quasi-para-Sasakian structure can be
obtained by a homothetic deformation of a para-Sasakian structure.}

\begin{proof}
First of all, $(i)$ implies $(ii)$ because of the dim$M=3$. From (\ref{C15}%
), one can see $(ii)\Leftrightarrow (iii)$. Now, we will show $(iii)$
implies $(iv)$. If $\beta $\textit{\ }is const.,\textit{\ }using (\ref{C15}%
), we get $\beta \left( 3\beta ^{2}+\frac{\tau }{2}\right) =0$ and $\tau $
is const. Now there are two possibilities. If $\beta =0$\textit{, }then $M$
is a paracosymplectic manifold which is locally a product of the real line $%
R $ and a $2$-dimensional para-Kaehlerian manifold. If $\beta \neq 0$, then $%
\tau =-6\beta ^{2}$, namely $M$ has constant negative curvature. By using $%
\tau =-6\beta ^{2}$ in (\ref{m2}), we get $M$ is Einstein since $S=\frac{%
\tau }{3}g$. Using$\ $Theorem \ref{mert}, one can say that the
quasi-para-Sasakian structure can be obtained by a homothetic deformation of
a para-Sasakian structure. One can easily deduce that $(iv)\Rightarrow (i)$.
\end{proof}

\end{document}